\documentclass[sn-mathphys-num]{sn-jnl}

\usepackage{epstopdf}

\usepackage{graphicx}%
\usepackage{multirow}%
\usepackage{amsmath,amssymb,amsfonts}%
\usepackage{amsthm}%
\usepackage{mathrsfs}%
\usepackage[title]{appendix}%
\usepackage{xcolor}%
\usepackage{textcomp}%
\usepackage{manyfoot}%
\usepackage{booktabs}%
\usepackage{algorithm}%
\usepackage{algorithmicx}%
\usepackage{algpseudocode}%
\usepackage{listings}%
\newtheorem{property}{Property}[section]
\newtheorem{lemma}{Lemma}[section]
\usepackage{quantikz}
\newcommand{\rinner}[1]{\mathfrak{Re}\left(\left<#1\right>\right)}
\newcommand{\inner}[1]{\left<#1\right>}

\usepackage{amsmath, amsthm, amssymb, subcaption, hyperref}	

\theoremstyle{thmstyleone}%
\newtheorem{theorem}{Theorem}
\theoremstyle{thmstyletwo}%

\theoremstyle{thmstylethree}%
\newtheorem{definition}{Definition}%

\begin{document}

\title[An Iterative Methodology for Unitary Quantum Channel Search]{An Iterative Methodology for Unitary Quantum Channel Search}


\author[1]{\fnm{Matthew M.} \sur{Lin}}\email{mhlin@mail.ncku.edu.tw}

\author[2]{\fnm{Hao-Wei} \sur{Huang}}\email{huanghw@math.nthu.edu.tw }

\author*[1]{\fnm{Bing-Ze} \sur{Lu}}\email{l18081028@gs.ncku.edu.tw}

\affil*[1]{\orgdiv{Department of Mathematics}, \orgname{National Cheng Kung University}, \orgaddress{\street{No.1, University Road}, \city{Tainan}, \postcode{70101}, \country{Taiwan}}}

\affil[2]{\orgdiv{Department of Mathematics}, \orgname{National Tsing Hua University}, \orgaddress{\street{101, Section 2, Kuang-Fu Road}, \city{Hsinchu}, \postcode{300}, \country{Taiwan}}}



\abstract{In this paper, we propose an iterative algorithm using polar decomposition to approximate a channel characterized by a single unitary matrix based on input-output quantum state pairs. In limited data, we state and prove that the optimal solution obtained from our method using one pair with a specific structure will generate an equivalent class, significantly reducing the dimension of the searching space. 
Furthermore, we prove that the unitary matrices describing the same channel differ by a complex number with modulus 1.
We rigorously prove our proposed algorithm can ultimately identify a critical point, which is also a local minimum of the established objective function.}




\maketitle

\section{Introduction}
A unitary quantum channel describes how a quantum state evolves in a closed quantum system while preserving coherence. Many studies of physical phenomena and circuit designs rely on closed quantum systems because they enable reversible and coherent quantum evolution. 
For instance, Schr\"{o}dinger's equation models the evolution of pure quantum states in such systems, while the Liouville-von Neumann equation extends this framework to mixed states \cite{Rivas12}. Additionally, any quantum circuit can be viewed as implementing a unitary quantum channel, as it operates through a sequence of unitary transformations \cite{Iten17}. Moreover, understanding the dynamics of a quantum mechanism or generating an executable quantum circuit that is sufficiently close to a quantum process is an important task \cite{Kunjummen23}. Beyond these scientific demands, a precise characterization of a quantum system can enhance understanding in many quantum applications, such as communication \cite{Gisin07}, control theory \cite{Dong10}, and information theory \cite{Petz86}.

There are various methods that focus on exploring a quantum channel from given input and output data. Readers can find \cite{Chuang97} for prescribing a quantum channel using quantum measurements. Besides, the authors \cite{Gross10} employ Pauli measurements and state the maximal requirement of the measurements to uniquely determine the desired unknown channel. In \cite{Bisio10}, the authors proposed a fidelity-based loss function to find the best unitary quantum channel that matches the input-output pairs. Later, in \cite{Kunjummen23}, a trace-distance-based shallow tomography method was introduced. In addition to traditional methods, many machine learning-based approaches have been developed to tackle such problems; for more details, readers are encouraged to explore \cite{Alvarez17}, \cite{Cemin24}, \cite{Kardashin22}, \cite{Huang22}, and others. 

Beyond algorithms, a notable analysis brought by \cite{Gutoski14} shows that there require $4n^2-2n-4$ Hermitian matrices to uniquely identify the Choi representation of a channel that maps $\mathbb{C
}^n$ to $\mathbb{C}^n$. Our research aims to propose an effective and efficient method that uses limited data to reach the unitary matrix that describes the unknown quantum channel to reduce the search space. To the best of our knowledge, the use of matrix optimization methods to solve the quantum channel identification problem has not yet been explored.

To begin, we introduce some notations to facilitate our later discussions. Let $\mathcal{S}_n$ denote the set of {$n$-by-$n$} unitary matrices:
\begin{equation*}
\mathcal{S}_n := \{U\in\mathbb{C}^{n\times n}: U^*U = I_n\},
\end{equation*}
where $U^*$ is the conjugate transpose of $ U $, and $I_n$ is the identity matrix of size $ n \times n$. Note that the dimension of $\mathcal{S}_n$ is $n^2$.

Given a sequence of Hermitian positive {\color{black}definite} matrices 
$(\rho, \sigma)$,
this paper focuses on solving the following optimization problem:

\begin{subequations}
\textbf{Problem 1}
\begin{eqnarray} 
\text{Minimize} && g(U) := \dfrac{1}{2}\displaystyle
\|\sigma - U \rho U^* \|_F^2,\\
\text{subject to} && U\in \mathcal{S}_n 
\end{eqnarray}
\end{subequations}\label{eq:prob1}
 where $ \|\cdot\|_F $ denotes the Frobenius norm. {Following the restriction stated in \eqref{eq:prob1}, the search space has dimension $\mathcal{O}(n^2)$. This presents a challenge in finding the optimal solution when $n$ is large enough.}

We aim to propose a convergent, first-order matrix algorithm that does not rely on massive amounts of data, unlike other approaches that require large datasets to achieve the approximation goal. Our method embeds the Stiefel manifold into standard Euclidean space and utilizes the polar decomposition to generate a convergent sequence, ensuring that the objective function decreases and ultimately approaches a local minimum solution.

Our contributions are four folds:
\begin{itemize}
    \item We state and prove in Theorem \ref{thm:global-theorem} that if a quantum channel can be expressed as 
    $\Phi(\rho) = U\rho U^*$
    where $\rho$ has non-degenerated eigenvalues, then all possible $U$ can form $\Phi$ are distinguished by a complex scalar of modulus 1.  
    \item Casting the problem to an optimization problem by defining an objective function constraining on the Stiefel manifold. We then propose an algorithm based on polar decomposition to reach an optimal solution.
    \item Rigorously prove the proposed algorithm monotonely decreases the objective value and converges to the set of critical points. These results build up a cornerstone for further reconstructing the unknown unitary channel using limited data. 
    {
    \item Propose a reconstruction process to recover the unknown $U$ which describes the channel $\Phi(\rho) = U\rho U^*$ for $U\in \mathbb{C}^{n\times n}$ using only $n^2+3n$ quantum observables and operations. A non-degenerate quantum state $\rho$ is the key ingredient in the procedure, which provides a sufficient condition allowing us to adopt the Theorem \ref{thm:one-theorem} to decrease the searching dimension to $n$.
    }
\end{itemize}

The remainder of the paper is organized as follows. In Section 2, we provide related preliminaries, including a precise characterization of the tangent space to $\mathcal{S}_n$, and state the properties of the optimal solutions using one-shot data and extra data. In Section 3, we present the polar decomposition-based algorithm to solve the optimization problem. Later on, in Section 4, we analyze the proposed method to ensure the loss function is decreasing following our method and prove the related convergence issue. In Section 5, we outline the reconstruction steps to recover $U$ from the solution computed by our proposed algorithm. We then assess the number of computational operations required. In Section 6, we verify our theoretical results by numerical examples, including a real-world application in which the objective is to recover an unknown quantum circuit. Finally, we close our paper in the conclusion section.

\section{Preliminaries}
In this section, we revisit essential concepts to support our upcoming discussion. Casting the optimization problem~\ref{eq:prob1} as the following problem.
\begin{subequations}
\begin{eqnarray} 
\text{Minimize} && f(X),\\
\text{subject to} && X\in \mathcal{S}_{n}. 
\end{eqnarray}
\label{eq:fundamental}
\end{subequations}

For any given complex matrix $X$, suppose that $X_{\mathfrak{Re}}$ and $X_{\mathfrak{Im}}$ represent the real and imaginary parts of $X$, respectively.
Let $\gamma:\mathbb{R} \rightarrow \mathcal{S}_n$ be a smooth mapping curve in $\mathcal{S}_n$ such that $\gamma(0) = X\in\mathcal{S}_n$.  Then 
\begin{equation}\label{eq:DFT}
\left.\dfrac{d f(\gamma(t))}{dt}\right\vert_{t=0}  =
 \operatorname{tr}\left(
 \left[
 \frac{\partial f(X)}{\partial X_{\!\mathfrak{Re}}}\right]^\top \dot{\gamma_{\mathfrak{Re}}}(0)
 \right
 )
 +
\operatorname{tr}\left( \left[
  \frac{\partial f(X)}{\partial X_{\!\mathfrak{Im}}}
 \right]^\top \dot{\gamma_{\mathfrak{Im}}}(0)
 \right)
\end{equation}
under the assumption of continuity of the function
\begin{large}$\frac{\partial f}{\partial X_{\!\mathfrak{Re}}}$\end{large} or \begin{large}$\frac{\partial f}{\partial X_{\!\mathfrak{Im}}}$\end{large}.
Here, the superscript "$\top$" denotes the transpose of the matrix. Since $\mathcal{S}_n$ is an embedded submanifold of the Euclidean space $\mathbb{C}^{n\times n}$~\cite{Shu2024}, i.e., $\gamma(t) \in \mathbb{C}^{n\times n}$, $\gamma(0)= X$, and $\gamma^*(t) \gamma(t) = I_n$ for all $t$, it follows by directly differentiating that 
\begin{equation}
\dot{\gamma}^*(t) \gamma(t) +\gamma(t)^* \dot{\gamma}(t) = 0.
\end{equation}
This implies that the tangent space to $\mathcal{S}_{n}$ at $X$, 
 denoted by $T_{X}\mathcal{S}_n$, is given by:
 \begin{eqnarray}\label{eq:tangent}
        T_X \mathcal{S}_n &=& \{Z \in\mathbb{C}^{m\times m}:
    X^* Z+ Z^* X  =0
    \} \nonumber \\
       &=& \{Z \in\mathbb{C}^{m\times m}:
    X^* Z \mbox{ is skew-Hermitian}
    \} \\
     &=&  X \mathcal{H}_n, \nonumber
\end{eqnarray}
 where $\mathcal{H}_n$ denotes the space of all $n \times n$ skew-Hermitian matrices (see also \cite{Edelman1998} for further details). Correspondingly, the orthogonal complement $N_{X} \mathcal{S}_n$ of  $T_{X}\mathcal{S}_n$ is given by: 
 \begin{equation}\label{eq:orthogonal}
        N_{X} \mathcal{S}_n =  X \mathcal{H}_n^\perp.
\end{equation}

To determine the direction of motion from a point $X\in\mathcal{S}_n$ that leads to the most significant decrease in the objective function~\eqref{eq:fundamental}, 
we use a real-valued inner product defined 
and its corresponding norm 
as follows:
\begin{eqnarray}\label{eq:inner}
    \langle A, B \rangle_{R} 
&=& \mbox{Re}(\operatorname{tr}(A^* B)),\\
\end{eqnarray}
where $A$ and $B$ are complex $n\times n$ matrices in $T_{X} \mathcal{S}_n$, 
$\mbox{Re}(\cdot)$ denotes the real part of a complex number, and $\operatorname{tr}(\cdot)$ denotes the trace of a square matrix.
Let $\nabla f(X)$ be the matrix defined by
\begin{equation*}
     \nabla  f (X) =  
\dfrac{\partial f(X)}{\partial X_{\mathfrak{Re}}}+i\dfrac{\partial f(X)}{\partial X_{\mathfrak{Im}}}.
\end{equation*}

The (normalized) direction of the steepest descent on $T_X\mathcal{S}_n$ is then given by 
\begin{equation}\label{eq:steep}
\xi_X=\underset{\xi\in T_X\mathcal{S}_n, \|\xi\|_X = 1}{\operatorname{argmin}} \langle \nabla f(X), \xi\rangle_R, 
\end{equation}
which simplifies to 
\[
\xi_X  = -\mathrm{Proj}_{T_X\mathcal{S}_n} (\nabla f(X)),
\]
where $\mathrm{Proj}_{T_X\mathcal{S}_n} (\nabla f(X))$ denotes the projection of  $\nabla f(X)$ onto the tangent space $T_X\mathcal{S}_n$.

 From~\eqref{eq:tangent} and~\eqref{eq:orthogonal}, we observe that any matrix $H\in\mathbb{C}^{n\times n}$ can be uniquely decomposed into 
 components in  $T_X\mathcal{U}_m$ and $N_X\mathcal{U}_m$ as follows: 
\begin{eqnarray}\label{eq:proj}
    H &=&
X\left\{ \frac{1}{2}(X^*H - H^* X)\right\}+X\left\{ \frac{1}{2}(X^* H + H^* X)\right\}\nonumber\\
&=& X \mathrm{skew}(X^* H) + X \mathrm{herm}(X^*H),
\end{eqnarray}
where $\mathrm{herm}(A)$ and $\mathrm{skew}(A)$ are defined as 
$\mathrm{herm}(A) :=\frac{1}{2}(A+A^*)$ and $\mathrm{skew}(A):= \frac{1}{2}(A-A^*)$, respectively.  Using~\eqref{eq:proj}, we can explicitly express the projection of $\nabla f(X)$ at $X$ onto $T_X\mathcal{S}_m$ as 
\begin{equation}\label{eq:gradF}
\mathrm{Proj}_{T_X\mathcal{S}_n} (\nabla f(X)) = X \mathrm{skew}(X^* \nabla f(X)). 
\end{equation}

Besides the geometric properties of the Stiefel manifold, we characterize the collection of solutions \eqref{eq:prob1} that result in zero objective function value forms an equivalent class. {To begin, we define the concept of a non-degenerate eigenvalue, which is essential for understanding the structure of unitary transformations.}

\begin{definition}
    An eigenvalue of a matrix $A$ is said to be \emph{non-degenerate} if its corresponding eigenspace is one-dimensional.
\end{definition}
{Additionally, we introduce a relation between unitary matrices through conjugation by a unitary matrix and multiplication with a diagonal matrix.
}
\begin{definition}
   Let $V\in \mathbb{C}^n$ and $VV^* = V^*V = I$. If we say $U_1\#_VU_2$, then there must exist a diagonal matrix $D = \mbox{diag}(e^{i\theta_1}\cdots e^{i\theta_n})$ such that 
   \begin{eqnarray*}
       U_1 = U_2 VDV^*,
   \end{eqnarray*}
   where $\theta_i$s are real values.
\end{definition}
{This relationship can be directly demonstrated to satisfy the properties of an equivalence relation.}
\begin{property}
    $\#_V$ is an equivalent relation.
\end{property}
\begin{theorem}\label{thm:one-theorem}
Let $\sigma$ and $\rho$ be Hermitian matrices with non-degenerate eigenvalues. Suppose there exists a subset $\mathcal{G}$ of unitary matrices  such that  
\[
\sigma = U \rho U^* \quad \text{for every } U \in \mathcal{G}.
\]
{
Let $V$ denote the collection of eigenvectors of $\rho$. Then, if the eigenvectors of $\rho$ and $\sigma$ are arranged in the same order as in $V$, the set $\mathcal{G}$ forms an equivalence class under the relation $\#_{V}$.
}
\end{theorem}
\begin{proof}
    Suppose there exist two unitary matrices $U_1$ and $U_2$ such that
    \begin{eqnarray*}
    \sigma &=& U_1 \rho U_1^* = U_2 \rho U_2^*.
    \end{eqnarray*}
    Then,
    \begin{eqnarray*}
    U_1 \rho U_1^* &=& U_2 \rho U_2^*.
    \end{eqnarray*}
    Multiplying on the left by $U_2^*$ and on the right by $U_1$ yields
    \begin{eqnarray*}
    U_2^* U_1 \rho &=& \rho\, U_2^* U_1.
    \end{eqnarray*}
    Since $\rho$ is Hermitian, it is diagonalizable. 
    {Let $(\lambda, v)$ be an eigenpair of $\rho$ satisfying $\rho v = \lambda v$. Then,
    \begin{eqnarray*}
    \rho\, (U_2^* U_1 v) = U_2^* U_1 \rho v = \lambda \, (U_2^* U_1 v).
    \end{eqnarray*}
    }
    This shows that $U_2^* U_1 v$ is also an eigenvector of $\rho$ corresponding to the eigenvalue $\lambda$. Because the eigenvalue $\lambda$ is non-degenerate, its eigenspace is one-dimensional, and hence
    \begin{eqnarray*}
    U_2^* U_1 v &=& e^{i\theta} v
    \end{eqnarray*}
    for some real $\theta$ that may depend on $v$. Since this holds for every eigenvector $v$ of $\rho$ (with the eigenvectors ordered consistently), the matrix $U_2^* U_1$ can be expressed as 
    \begin{eqnarray*}
        U_2^* U_1 V = V D \Rightarrow U_1 = U_2 V D V^*,
    \end{eqnarray*}
    where $D = \mbox{diag}(e^{i\theta_1}\cdots e^{i\theta_n})$. 
{Thus, we demonstrate that if $U_1$ and $U_2$ are elements of $\mathcal{G}$, then $U_1 \#_V U_2$.}
\end{proof}

Finally, we now show that if a channel can be represented as
    \begin{eqnarray*}
        \Phi(\rho)=U\rho U^*,
    \end{eqnarray*}
    by a unitary $U$, this representation is unique up to a global phase. In other words, if two unitary matrices $U$ and $V$ define the same channel, they differ only by a complex scalar of modulus one.
    \begin{theorem}\label{thm:global-theorem}
        Suppose $U$ and $V$ are unitary matrices such that
        \begin{eqnarray*}
            U\rho U^* = V\rho V^*
        \end{eqnarray*}
        for every Hermitian matrix $\rho$ with $\mbox{Tr}(\rho)=1$. Then, there exists a scalar $\mu$ with $|\mu|=1$ such that
        \begin{eqnarray*}
            U = \mu V.
        \end{eqnarray*}
    \end{theorem}
    \begin{proof}
        Since
        \begin{eqnarray*}
        U\rho U^* = V\rho V^*
        \end{eqnarray*}
        for all such $\rho$, define the matrix $M = V^* U$, then we have  $M\rho = \rho M$ 
        for every Hermitian matrix $\rho$ with its trace equal to 1.
        
        Choose the standard basis $\{e_i\}_{i=1}^n$ of $\mathbb{R}^n$. Any Hermitian matrix $\rho$ with trace $1$ can be expressed on this basis as
        \begin{eqnarray*}
        \rho = \sum_{i=1}^n \rho_{ii}\, e_i e_i^T + \sum_{i<j} \rho_{ij}\, (e_i e_j^T + e_j e_i^T),
        \end{eqnarray*}
        with $\sum_{i=1}^n \rho_{ii}=1$ and each $\rho_{ij}\in\mathbb{R}$.
        Express $M$ in the same basis:
        \begin{eqnarray*}
        M = \sum_{p,q=1}^n M_{pq}\, e_p e_q^T.
        \end{eqnarray*}
        Now, the commutation relation $M\rho = \rho M$ must hold for all such $\rho$. First, we are allowed to consider the effect on the matrices $e_i e_i^T$:
        \begin{eqnarray*}
        M\, e_i e_i^T = e_i e_i^T\, M.
        \end{eqnarray*}
        A straightforward calculation using the above representations shows that this forces $M_{pq} = 0$ whenever $p \neq q$; that is, $M$ must be diagonal.\\
        Next, look at the matrices $e_i e_j^T + e_j e_i^T$ for $i\neq j$. The relation
        \begin{eqnarray*}
        M(e_i e_j^T+e_j e_i^T) = (e_i e_j^T+e_j e_i^T) M
        \end{eqnarray*}
        implies that the diagonal entries of $M$ satisfy $M_{ii} = M_{jj}$ for all $i, j$.\\
        Finally, because $U$ and $V$ are unitary, so is $M = V^* U$. Hence, we have
        \begin{eqnarray*}
        MM^* = I,
        \end{eqnarray*}
        which implies that each diagonal entry of $M$ must have absolute value one; that is, $|M_{ii}|=1$.\\
        Since $M$ is diagonal with all entries equal to the same complex number of modulus one, we can write $M=\mu I$ with $|\mu|=1$. Therefore,
        \begin{eqnarray*}
        V^* U = \mu I \quad \Longrightarrow \quad U = \mu V,
        \end{eqnarray*}
        completing the proof.
    \end{proof}

%
%
%

\section{Iterative methodology}
To efficiently compute the optimal solution, it is essential to introduce a critical result that forms the foundation of our subsequent argument. For clarity in the following discussion, we also use $\mathfrak {Re}(X)$ and $\mathfrak{Im}(X)$ to denote the real and imaginary parts of any matrix $X$, respectively.

\begin{lemma}\label{lemma:derivative}
 Given two Hermitian matrices $A$  and $B$, let $h:\mathbb{C}^{n\times n} \rightarrow \mathbb{R}$ be a function denoted by
\begin{equation*}
 h(X) = \mathfrak{Re}(\inner{A, XBX^*}).
\end{equation*}
Then the partial derivatives of $h$ with respect to variables $X_{\mathfrak{Re}}$ and $X_{\mathfrak{Im}}$  are  
\begin{equation}
\left\{ \begin{array}{rcl}
  \dfrac{\partial h}{\partial X_{\mathfrak{Re}}} &=& 
  2 \mathfrak{Re}(AXB),\\[0.3cm]
        \dfrac{\partial h}{\partial X_{\mathfrak{Im}}} &=& 
        2 \mathfrak{Im}(AXB).
\end{array} \right. 
 \label{eq:gradient_h}
\end{equation} 
\end{lemma}
\begin{proof}

Rewrite $h(X)$ as
\begin{displaymath}
h(X)= \frac{1}{2}  \left(\langle  A, \overline{X}\overline{B}X^\top \rangle_{\mathbb{R}} + \langle  XBX^*, \overline{A}
 \rangle_{\mathbb{R}} \right) 
\end{displaymath}
where
\begin{equation*}
\langle X,Y \rangle_{\mathbb{R}}:=\sum_{i,j=1}^{n}
x_{ij}y_{ij} \label{eq:real_inner_product}
\end{equation*}
stands for the formal inner product for two matrices $X$ and $Y$ over the real field.

Considering $X$ and $\overline{X}$ as independent variables,
we  formally take the Fr\'{e}chet derivative of $h$ as an action on an arbitrary ${\Delta X} \in \mathbb{C}^{n\times n}$. Denoting this action by $.$, we obtain by direct computation that 
\begin{eqnarray*}
\frac{\partial h}{\partial X}.{\Delta X} &=&
\langle \overline{A}\overline{X}\overline{B},  \Delta X\rangle_{\mathbb{R}},  \\
\frac{\partial h}{\partial \overline{X}}.{\Delta X} &=&
\langle A X B,  \Delta X\rangle_{\mathbb{R}}. \label{eq:Frechet}
\end{eqnarray*}
According the Wirtinger calculus, the partial derivatives of $h$ with respect to $X_{\mathfrak{Re}}$  and $X_{\mathfrak{Im}}$ can be evaluated  as
\begin{eqnarray*}
\frac{\partial h}{\partial X_{\mathfrak{Re}}} &=& \frac{\partial h}{\partial X} + \frac{\partial h}{\partial \overline{X}} = 2 \mathfrak{Re}(AXB),\\
\frac{\partial h}{\partial X_{\mathfrak{Im}}} &=& i\left(\frac{\partial h}{\partial X} - \frac{\partial h}{\partial \overline{X}}\right) = 2 \mathfrak{Im}(AXB). \label{eq:Frechet2}
\end{eqnarray*}
\end{proof}

Since $g(U) =\dfrac{1}{2}\displaystyle
\|\sigma - U \rho U^* \|_F^2$ is the objective function in Problem~\ref{eq:prob1}. Through direct computation, we obtain 
\begin{equation}\label{eq:objg}
g(U) = \frac{1}{2}(\|\sigma\|_F^2+\|\rho\|_F^2 - {\color{black}{2}}\mathfrak{Re}(\inner{\sigma, U\rho U^*}) ). 
\end{equation} 
Using the result from Lemma~\ref{lemma:derivative}, we can derive the Euclidean gradient of the objective function $g(U)$ as follows:
\begin{theorem}
    In Problem~\ref{eq:prob1} , the Euclidean derivative of objective function $g(U)$ 
    with respect to real and imaginary components of    
    $U$ admit the following form: 
    {\color{black}
    \begin{equation}\label{eq:par2}
        \left\{
            \begin{array}{rl}
                \dfrac{\partial g}{\partial U_{\mathfrak{Re}}} &=  - {\color{black}{2}}\mathfrak{Re}\left(\sigma U \rho\right),
                \\
                \dfrac{\partial g}{\partial U_{\mathfrak{Im}}} &= - {\color{black}{2}}\mathfrak{Im}\left(\sigma U \rho\right).
            \end{array}
        \right.
    \end{equation}
    }
\end{theorem}

{\color{black}
Since the objective function $g$ is constrained to the set of unitary matrices $\mathcal{S}_n$, the first-order optimality condition for the approximation problem~\ref{eq:prob1} can be derived by applying by projecting the Euclidean gradient onto $\mathcal{S}_n$. The following theorem shows that one can obtain the critical point of $g$ by simply applying the polar decomposition of $-\nabla g$ given by
\begin{equation*}
     -\nabla  g (U) = 
\dfrac{\partial g(U)}{\partial U_{\mathfrak{Re}}}+i\dfrac{\partial g(U)}{\partial U_{\mathfrak{Im}}} =  {\color{black}{2}} \sigma U \rho.  
\end{equation*}
\begin{theorem}\label{lemma:polarcri}
 {Let $g: \mathcal{S}_n \rightarrow \mathbb{R}$ be the objective function in Problem~\ref{eq:prob1} and $U$ be a unitary matrix.  Assume that a Hermitian positive definite matrix $P$ exists such that $UP = -\nabla g(U)$. Then} $U$ is a critical point of the function $g$ constrained to $\mathcal{S}_n$.
\end{theorem}
\begin{proof}
It follows from~\eqref{eq:proj} that the projection of $-\nabla g(U)$ onto $\mathcal{S}_n$ can be computed as
\begin{equation*}
U\left\{ \frac{1}{2} \left(U^* \left(-\nabla g(U)\right) - \left(-\nabla g(U)\right)^* U \right)\right\}.
\end{equation*}
Given the condition $UP = -\nabla g(U)$, we have
    \begin{equation*}
        U^*\left(-\nabla g(U)\right)- \left(-\nabla g(U)\right)^* U = -P+(UP)^* U = 0.
    \end{equation*}
This confirms that if $U$ satisfies $UP = -\nabla g(U)$, then $U$ is a critical point of the objective function $g$.
\end{proof}

%
%

Based on Theorem~\ref{lemma:polarcri}, we propose an algorithm to produce a convergent sequence to determine a critical point of the objective function $g$ on $\mathcal{S}_n$. 
By employing the notation ``\emph{poldec}" to denote the polar decomposition of a given matrix, we present the iterative procedure in Algorithm~\ref{alg:poldec}, which systematically refines the solution through successive iterations. The convergence properties of the algorithm, along with the conditions that guarantee convergence to an optimal solution, will be discussed in detail in the next section.
}

\begin{algorithm}[!htt]
\caption{}\label{alg:poldec}
\begin{algorithmic}
\Require A pair of Hermitian positive {\color{black}definite} matrices $(\sigma, \rho)$ and an initial unitary matrix $U^{(0)}$. A positive number $M$ to terminate the algorithm.
\Ensure: {\color{black}An approximation of the quantum state $ \sigma$ after the unitary channel $U$ is applied to the state $\rho$.}

\\\hrulefill
\For {$s = 0, 1 \cdots , M$}
    
    \State $[U^{(s+1)}, P^{(s+1)}] = {\color{black}
    \mbox{{poldec}}\left( 2\sigma U^{(s)}\rho  \right)
    }$. 
    \\
    \If {$\dfrac{1}{2}\left\|\sigma - U^{(s+1)}\rho \left(U^{(s+1)}\right)^*\right\|_F^2<TOL$}
    \State \textbf{Break}
    \EndIf
\EndFor
\end{algorithmic}
\end{algorithm}


\section{Convergence Analysis}

To ensure that the proposed algorithm effectively decreases the objective function 
\[
g(U) = \frac{1}{2} \left\|\sigma - U \rho U^*\right\|_F^2,
\]
and converges uniquely to an optimal unitary matrix $\hat{U}$ within the set of unitary matrices $\mathcal{S}_n$, we must consider a crucial property related to polar decomposition. This property demonstrates that for any matrix $X \in \mathbb{C}^{n \times n}$, its polar decomposition provides the closest matrix in $ \mathcal{S}_n$ to itself.

\begin{lemma}\label{lem:polar}~\cite[Corollary 2.3]{Higham86}
For any matrix $X \in \mathbb{C}^{n \times n}$, let $X = UP$  be its polar decomposition, where $U \in \mathcal{S}_n$ and $P \in \mathbb{C}^{n \times n}$ is Hermitian and positive definite. Then, $U$ is the matrix in $ \mathcal{S}_n$ that is closest to $X$, i.e., for any $Z \in \mathcal{S}_n$,
\begin{equation}\label{eq:polardistance}
\|X - U\|_F \leq \|X - Z\|_F.
\end{equation}
\end{lemma}

We then apply Lemma~\ref{lem:polar} to establish the decreasing behavior of $g(U)$ at each algorithm iteration.
\begin{theorem}\label{thm:decreasing}
The sequence generated by Algorithm~\ref{alg:poldec} ensures a non-increase in the value of the function $g(U)$ at each iteration. 
\end{theorem}
\begin{proof}
For $s\geq 1$, the matrix $U^{(s+1)}$, generated by Algorithm~\ref{alg:poldec} follows from the polar decomposition, which ensures that $U^{(s+1)}$ is the closest unitary matrix to the given matrix at each iteration. Hence, the update at each step
 %
    \begin{equation*}
    \|\sigma U^{(s)}\rho - U^{(s+1)}\|_F^2 \leq  \|\sigma U^{(s)}\rho - U^{(s)}\|_F^2,
    \end{equation*}
or equivalently, 
\begin{equation*}
-\rinner{\sigma, U^{(s+1)}\rho \left(U^{(s)}\right)^*}
\leq 
-\rinner{\sigma, U^{(s)}\rho \left(U^{(s)}\right)^*}.
\end{equation*}
Therefore, we have
    \begin{eqnarray*}
        g(U^{(s)}) &=& 
        \dfrac{1}{2} \left(\|\sigma\|_F^2+\|\rho\|_F^2-2\rinner{\sigma, U^{(s)}\rho \left(U^{(s)}\right)^*}\right)\\
        &\geq&\dfrac{1}{2} \left(\|\sigma\|_F^2+\|\rho\|_F^2-2\rinner{\sigma, U^{(s+1)}\rho \left(U^{(s)}\right)^*}\right).
    \end{eqnarray*}
Let $\Delta U = U^{(s+1)}-U^{(s)}$ denote the update step. 
Since both $\rho$ and $\sigma$ are positive definite matrices, we can assert that $\inner{\sigma, \Delta U\rho \left(\Delta U\right)^*} \geq 0$. This follows from the fact that the inner product of Hermitian positive semidefinite matrices is nonnegative. Moreover,  from Algorithm~\ref{alg:poldec}, we know that  \[U^{(s+1)}P^{(s+1)} = 2\sigma U^{(s)}\rho.\] Thus, the update step $\Delta U$ is well-defined and satisfies the relationship
    \begin{eqnarray*}
        &&\rinner{\sigma, U^{(s)}\rho \left(\Delta U\right)^*} \\&&= \mathfrak{Re}\left(Tr\left(\sigma U^{(s)}\rho \left(U^{(s+1)}\right)^*\right)\right) - \mathfrak{Re}\left(Tr\left(\sigma U^{(s)}\rho \left(U^{(s)}\right)^*\right)\right)\\
        &&=  \dfrac{1}{2}        \mathfrak{Re}\left(Tr\left(U^{(s+1)}P^{(s+1)} \left(U^{(s+1)}\right)^*\right)\right) -\dfrac{1}{2}
    \mathfrak{Re}\left(Tr\left(U^{(s+1)}P^{(s+1)} \left(U^{(s)}\right)^*\right)\right)\\
        &&=\dfrac{1}{2}
     \mathfrak{Re}\left(Tr\left(P^{(s+1)}\left(I_n-\left(U^{(s)}\right)^*
       U^{(s+1)}
       \right)  \right)\right)
       .
    \end{eqnarray*}
    Because the matrix $I-\left(U^{(s)}\right)^*U^{(s+1)}$ is positive semidefinite, and the trace of the product of two positive semidefinite matrices is nonnegative, we have
    \begin{equation*}
        \rinner{\sigma, U^{(s)}\rho \left(\Delta U\right)^*} \geq 0.
         \end{equation*}

Next, observe that 
    \begin{eqnarray*}
        &&\rinner{\sigma, U^{(s+1)}\rho \left(U^{(s+1)}\right)^*} -\rinner{\sigma, U^{(s+1)}\rho \left(U^{(s)}\right)^*} \\
        &=&  \rinner{\sigma, \Delta U\rho \left(\Delta U\right)^*}+\rinner{\sigma, U^{(s)}\rho \left(\Delta U\right)^*}.
    \end{eqnarray*}
Since both terms on the right-hand side are nonnegative, it follows that 
$\rinner{\sigma, U^{(s+1)}\rho \left(U^{(s+1)}\right)^*} \geq \rinner{\sigma, U^{(s+1)}\rho \left(U^{(s)}\right)^*}$ and subsequently, the function $g$ satisfies the decreasing property 
    \begin{eqnarray*}
        g(U^{(s)}) 
        &\geq&\dfrac{1}{2} \left(\|\sigma\|_F^2+\|\rho\|_F^2-2\rinner{\sigma, U^{(s+1)}\rho \left(U^{(s)}\right)^*}\right)\\
        &\geq&\dfrac{1}{2} \left(\|\sigma\|_F^2+\|\rho\|_F^2-2\rinner{\sigma, U^{(s+1)}\rho \left(U^{(s+1)}\right)^*}\right) = g(U^{(s+1)}),
    \end{eqnarray*}    
    i.e., the objective function does not increase between iterations, which completes the proof.
\end{proof}

Note that in Theorem~\ref{thm:decreasing}, we have shown that the objective value is decreasing along the sequence generated by Algorithm~\ref{alg:poldec} unless it attains its local minimum points. Our next focus is to investigate the convergence of the sequence $ U^{(s)}$. To achieve this, we need to establish the interrelationship governing the updates.
\begin{lemma}\label{lemm:sub}
    The sequence $\left\{U^{(s)}\right\}_{s=1}^\infty$ generated by algorithm \ref{alg:poldec} satisfies $\|U^{(s+1)}-U^{(s)}\|_F^2\rightarrow 0 $ as $s\rightarrow\infty$.
\end{lemma}
\begin{proof}
Since $g$ is non-increasing along the sequence $\left\{ U^{(s)} \right\}$, 
for any given positive number $\epsilon$, 
we can select a positive integer $s_0(\epsilon)$ such that $s_0(\epsilon)$ is the smallest integer for which $g(U^{(s)}) - g(U^{(s+1)}) \leq \epsilon$.

Next, we aim to show that
$\|U^{(s)}-U^{(s+1)}\|\leq \epsilon$ for $s>s_0(\epsilon)$. Note that 
{\footnotesize
 \begin{eqnarray*}
        g(U^{(s)}) - g(U^{(s+1)}) &=& -\rinner{\sigma, U^{(s)}\rho \left(U^{(s)}\right)^*}+\rinner{\sigma, U^{(s+1)}\rho \left(U^{(s+1)}\right)^*}\\
        &=&  -\rinner{\sigma, U^{(s)}\rho \left(U^{(s)}\right)^*}+\rinner{\sigma, U^{(s+1)}\rho \left(U^{(s+1)}\right)^*}\\
        && +\rinner{\sigma, U^{(s)}\rho \left(U^{(s+1)}\right)^*}-\rinner{\sigma, U^{(s)}\rho \left(U^{(s+1)}\right)^*}\\
        &=& -\rinner{\sigma, \left(U^{(s)}+U^{(s+1)}\right)\rho \left(U^{(s)}-U^{(s+1)}\right)^*}\\
        &=& \rinner{\sigma, \left(\Delta U+2U^{(s)}\right)\rho \left(\Delta U\right)^*}\geq \rinner{\sigma, \left(\Delta U\right)\rho \left(\Delta U\right)^*},
\end{eqnarray*}
}
The third equality holds from following the equality: 
{
\begin{equation*}
\inner{A, B} = \operatorname{tr}(AB^*) = \overline{\operatorname{tr}(A^*B)} = \overline{\inner{A^*, B^*}},
\end{equation*}
}
i.e., 
\begin{eqnarray*}
\rinner{\sigma, U^{(s)}\rho \left(U^{(s+1)}\right)^*} = \rinner{\sigma, U^{(s+1)}\rho \left(U^{(s)}\right)^*}.
\end{eqnarray*}
since $\sigma$ and $\rho$ are Hermitian matrices. 

To proceed, we first consider the diagonalization of the Hermitian positive definite matrices $\rho$ and $\sigma$. 
Since $ \rho $ and $ \sigma $ are Hermitian, they can be diagonalized as follows
\begin{equation}\label{eq:diag}
\left \{
\begin{array}{l}
\rho = Q_\rho \Lambda_\rho Q_\rho^*, \\
\sigma = Q_\sigma \Lambda_\sigma Q_\sigma^*,
            \end{array}
        \right.
    \end{equation}
where 
$Q_\rho$ and $Q_\sigma$ are unitary matrices whose columns are the eigenvectors of the Hermitian matrices $\rho$ and $\sigma$, respectively, while $ \Lambda_\rho$ and $\Lambda_\sigma$ are diagonal matrices containing their corresponding non-negative eigenvalues.

By applying the decomposition in~\eqref{eq:diag}, we obtain the following expression
\begin{eqnarray*}\label{eq:eval}
        \rinner{\sigma, \left(\Delta U\right)\rho \left(\Delta U\right)^*}
&=&\mathfrak{Re}(Tr\left( Q_\rho \Lambda_\rho Q_\rho^* (\Delta U) Q_\sigma \Lambda_\sigma Q_\sigma^* (\Delta U)^*\right)) \nonumber\\
&=& \mathfrak{Re}\left(Tr\left( (\Lambda_\sigma Q_\sigma^* (\Delta U)^* Q_\rho )(\Lambda_\rho Q_\rho^* (\Delta U)^* Q_\sigma )\right)\right).
    \end{eqnarray*}

Let $R = Q_\sigma^* (\Delta U)^* Q_\rho$. Then, the equation above implies  
    \begin{eqnarray*}
        \mathfrak{Re}\left(Tr\left( (\Lambda_\sigma R)(\Lambda_\rho R^*)\right)\right) &=& \mathfrak{Re}\left(\sum_{i}\sum_{j} (\Lambda_\sigma R)_{ij} (\Lambda_\rho R^*)_{ji}\right) \\
        &=& \mathfrak{Re}\left(\sum_{i}\sum_{j} ((\Lambda_\sigma)_{ii} R_{ij} ((\Lambda_\rho)_{jj} R^*)_{ji}\right)\\
        &\geq& \min_j\left\{(\Lambda_\sigma)_{jj}^2,(\Lambda_\rho)_{jj}^2 \right) \sum_{i, j} R_{ij}R^*_{ji} \geq c \|R\|_F^2, 
    \end{eqnarray*}
  for some $c>0$.
{
  The third inequality follows from the fact that both $ \rho $ and $ \sigma $ have eigenvalues that are positive and strictly less than one, and that for any two positive numbers, their product is greater than or equal to the smaller of their squares.
  }
Since $\|R\|_F^2 = \|\Delta U\|_F^2$, this completes the proof.
\end{proof}

So far, we have not established the convergence of the sequence. However, since the generated sequence $\{U^{(s)}\}$ is bounded, we can conclude that there exists a convergent subsequence, which satisfies the following property.

\begin{lemma}\label{lemma:subseq}
        
    {\color{black}If $\nabla g(U)$ is nonsingular, then any convergent subsequence generated by Algorithm~\ref{alg:poldec} has a limit point $\hat{U}$} that satisfies the condition
    \begin{eqnarray*}
        \hat{U} = \mbox{poldec}\left(-\nabla g(\hat{U})\right),
    \end{eqnarray*}
    where $\hat{U}$ 
     represents the unitary component obtained from the polar decomposition of $-\nabla g(\hat{U})$. 
\end{lemma}


To establish the convergence of the sequence ${U^{(s)}}$, we begin by introducing three essential lemmas. These lemmas will help us analyze and demonstrate the convergence behavior of the sequence generated by Algorithm~\ref{alg:poldec}. The first lemma addresses the general properties of iterative sequences and provides insight into the potential convergence behavior when continuous mappings produce such sequences. This important finding will help us understand how the sequence changes over time through iterations.
\begin{lemma}\label{lemma:Guan18}\cite[Theorem 1]{Guan18}
Let $\mathcal{F}: M \rightarrow M$ be a continuous map over a compact subset $M$ of a finite-dimensional Euclidean space. Given an initial point $U^{(0)}\in U$, consider the sequence $\{U^{(s)}\}$ generated by the iterative scheme
\[
U^{(s+1)} = \mathcal{F}(U^{(s)}), \quad s = 0, 1, 2, \ldots
\]
which is assumed to be well-defined. 
If the sequence $\{U^{(s)}\}$ is bounded and has only finitely many accumulation points, then the following holds:
\begin{enumerate}
    \item The sequence $\{U^{(s)}\}$ converges, or
    \item For sufficiently large $s$, the consecutive elements $U^{(s)}, U^{(s+1)}, \dots$ exhibit cyclic behavior. Specifically, disjointed neighborhoods exist around the accumulation points, so the sequence cyclically visits each neighborhood.
\end{enumerate}
\end{lemma}

%
Additionally, the second lemma, as stated in~\cite[Theorem 7.1.1]{Sommese05}, addresses the number of solutions to polynomial systems. This result characterizes the behavior of the critical points of an objective function, where the solutions of a polynomial system define the first-order optimality conditions.
\begin{lemma}\label{lemma:count}
 Let $P(z; q): \mathbb{C}^n \times \mathbb{C}^m \to \mathbb{C}^n$ be a system of polynomials in $n$ variables and $m$ parameters; that is,
 $F(z; q) = \{f_1(z; q), \ldots, f_n(z; q)\}$ and each $f_i(z; q)$ is polynomial in both $z$ and $q$ . 
Furthermore, let $\mathcal{N}(q)$ denote the number of nonsingular solutions as a function of $q$:
\[
 \mathcal{N}(q) := \# \left\{ z \in \mathbb{C}^n \mid F(z; q) = 0, \det \left( \frac{\partial F}{\partial z}(z; q) \right) \neq 0 \right\}.
 \]
Then, 
 \begin{enumerate}
\item $\mathcal{N}(q)$ is finite and is the same, denoted as $\mathcal{N}$, for almost all $q \in \mathbb{C}^m$;
\item For all $q \in \mathbb{C}^m$, it holds that $\mathcal{N}(q) \leq \mathcal{N}$;


 \end{enumerate} 
\end{lemma}







Finally, the third lemma in real analysis provides a valuable property for analyzing the convergence of the sequence.
\begin{lemma}\label{lemma:conv}\cite[Lemma 4.4]{Chu93}
 Let $\{u^{(s)}\}$ be a bounded sequence of real numbers such that $|u^{(s+1)} - u^{(s)}| \to 0$ as $s \to \infty$. If the sequence has only finitely many limit points, then $\{u^{(s)}\}$ converges to a single, unique limit point.
\end{lemma}


We now present the convergence result of the proposed algorithm.
\begin{theorem}
    The sequence generated by Algorithm~\ref{alg:poldec} with any unitary matrix $U^{(0)}\in \mathbb{C}^{n\times n}$ has a unique limit point, which is also a critical point of { \color{black}Problem~\ref{eq:prob1}} for almost all pairs of Hermitian positive definite matrices $(\sigma, \rho)$. 
\end{theorem}
\begin{proof}
Note that the set $\mathcal{S}_n$ is compact and bounded, meaning any sequence $U^{(s)} \in \mathcal{S}_n$ must have a convergent subsequence. Since the polar decomposition is unique for full-rank matrices and continuous concerning its parameters, 

According to Lemma~\ref{lemma:subseq}, all accumulation points satisfy the equation
\[
\hat{U}^{(s)} \hat{P}^{(s+1)} = -\nabla g(\hat{U}^{(s)})
\]
for some Hermitian positive definite matrix $\hat{P}^{s+1}$. Additionally, by Lemma~\ref{lemma:polarcri}, $\hat{U}^{(s)}$ are critical points of \textcolor{black}{Problem~\ref{eq:prob1}} and solutions to the first-order condition of the objective function $g$ subject to the unitary constraint. {\color{black}This condition is a polynomial equation.} From Lemma~\ref{lemma:count}, it follows that the number of critical points is finite for almost all pairs of $(\sigma, \rho)$.

Therefore, based on  Lemma~\ref{lemma:Guan18}, we conclude that the sequence $\{U^{(s)}\}$ either converges or forms a cyclic orbit. Finally, Lemma~\ref{lemma:conv} ensures convergence; otherwise, the distance $\|U^{(s)} - U^{(s+1)}\|_F^2$ would not converge to zero. This completes the proof.
%
%
\end{proof}


%
%

\section{Reconstruction Process and Complexity Analysis}
First, we introduce the naive quantum state tomography process to a quantum state $\rho\in \mathbb{C}^{n\times n}$. We write $\rho = \sum_{i,j= 1}^n \rho_{ij} e_i e_j^T$, $\rho_{ij} \in \mathbb{C}$, where $e_i$ is a standard vector in $\mathbb{C}^n$ taking value 1 at the $i$-th position.
Define two Hermitian matrices, $\left(E_{ij}\right)_+ = (e_i e_j^T+e_j e_i^T)/2$, and $\left(E_{ij}\right)_- = (e_i e_j^T-e_j e_i^T)/{2i}$. The standard quantum state tomography process utilizes $\left\{\left(E_{ij}\right)_+, \left(E_{ij}\right)_-\right\}_{i,j = 1}^n$to identify the  $\rho_{ij}$ by computing $\mbox{Tr}(\rho \left(E_{ij}\right)_+)$ and $\mbox{Tr}(\rho \left(E_{ij}\right)_-)$ to obtain the real and imaginary part of $\rho_{ij}$, respectively. Thus, a naive quantum state tomography requires $n^2+n$ operators to determine the description of $\rho$ uniquely.

Next, we analyze the total operations of applying our proposed method to reconstruct the quantum channel prescribed on unitary matrices. Leveraging the proposed algorithm alongside {Theorem \ref{thm:one-theorem} and Theorem \ref{thm:global-theorem}}, we outline the following procedure to reconstruct the quantum channel 
$\Phi(\rho) = U \rho U^*$ from given quantum states $\rho$.

\paragraph{Step 1}
Choose a non-degenerate state $\rho_0$ and order its eigenvectors in a fixed order to form the matrix $V = [v_1, v_2, \cdots v_n]$. Then, pass this $\rho_0$ to a channel $\Phi$ in interest. We must apply quantum state tomography to $\Phi(\rho_0)$ to prepare information for the proposed algorithm. Finally, we obtain a unitary estimate $U_0$ given from our method.

\paragraph{Step 2}
{
From the theoretical result of {Theorem} \ref{thm:one-theorem}, there exists a diagonal matrix $D$ whose diagonal entries have unit modulus, such that $U_0 = U V D V^*$. Consequently, we obtain  
\[
\Phi(\rho) = U_0 \bigl(V D V^*\bigr) \rho \bigl(V D^* V^*\bigr) U_0^*.
\]
}

\paragraph{Step 3}
Consider two specific quantum states:  
\[
\left(\rho_{p,q,r}\right)_+ = v_r v_r^* + \dfrac{1}{2}(v_p v_q^* + v_q v_p^*),
\]
and  
\[
\left(\rho_{p,q,r}\right)_- = v_r v_r  ^* + \frac{1}{2i} (v_p v_q^* - v_q v_p^*),
\]
where $p \neq q$ and $p, q \neq i$.
Without loss of generality, we consider the case where $r = 1$. By expressing $\Phi\left(\left(\rho_{p,q,r}\right)_+\right)$ and $ \Phi\left(\left(\rho_{p,q,r}\right)_-\right)$ in the following form:  
\begin{align*}
\begin{array}{l}
     U_0\left(\sum_{i=1}^n d_{ii} v_i v_i^* \left(v_1 v_1^* +\frac{1}{2} (v_p v_q^* + v_q v_p^*)\right) \sum_{j=1}^n d_{jj}^* v_j v_j^*\right)  U_0^*,\\
     U_0\left(\sum_{i=1}^n d_{ii} v_i v_i^* \left(v_1 v_1^* + \frac{1}{2i} (v_p v_q^* - v_q v_p^*)\right) \sum_{j=1}^n d_{jj}^* v_j v_j^*\right)  U_0^*.
\end{array}
\end{align*}
We can identify $(d_{pp}d_{qq}^*)$ by computing $u_1^*\Phi\left(\left(\rho_{p,q,1}\right)_+\right)u_1$ and $u_1^*\Phi\left(\left(\rho_{p,q,1}\right)_-\right)u_1$. {Therefore, two computational steps are required to obtain $d_{pp}d_{qq}^*$.

\paragraph{Step 4}
Repeat Step 3 until each product $d_{pp}d_{qq}^*$, for all $p \neq q$, is determined. The total computational cost required is  $2n$, including the real part and imaginary part.} Finally, by fixing $d_{11} = 1$ without loss of generality, the remaining diagonal elements $d_{22}, \dots, d_{nn}$ can be uniquely reconstructed.

In summary, the reconstruction process requires $n^2+n+2n$ tomography steps.

\section{Numerical Experiments}
In this section, we present two examples to demonstrate the effectiveness of our proposed algorithm and the convergence of the sequence generated by the algorithm.

\textbf{Example 1}
In this example, we will demonstrate the effectiveness of our proposed algorithm for exploring the unitary quantum channel that sends a given quantum state $\rho$ to an observation state $\sigma$. We generate the original unitary channel $U\in\mathbb{C}^{10\times 10}$ using the SVD algorithm applied to a randomly generated Hermitian positive definite matrix. The quantum state $\rho\in\mathbb{C}^{10\times 10}$ will also be a randomly generated Hermitian positive definite matrix. The observation state $\sigma$ will be calculated as $\sigma = U\rho U^*$. We then apply our proposed algorithm to approximate the channel and record the value of the objective function as well as the difference in the iterative $U^{(s)}$.
\begin{figure}[!ht]
    \begin{subfigure}[t]{.5\textwidth}
        \centering
        \includegraphics[width=\linewidth]{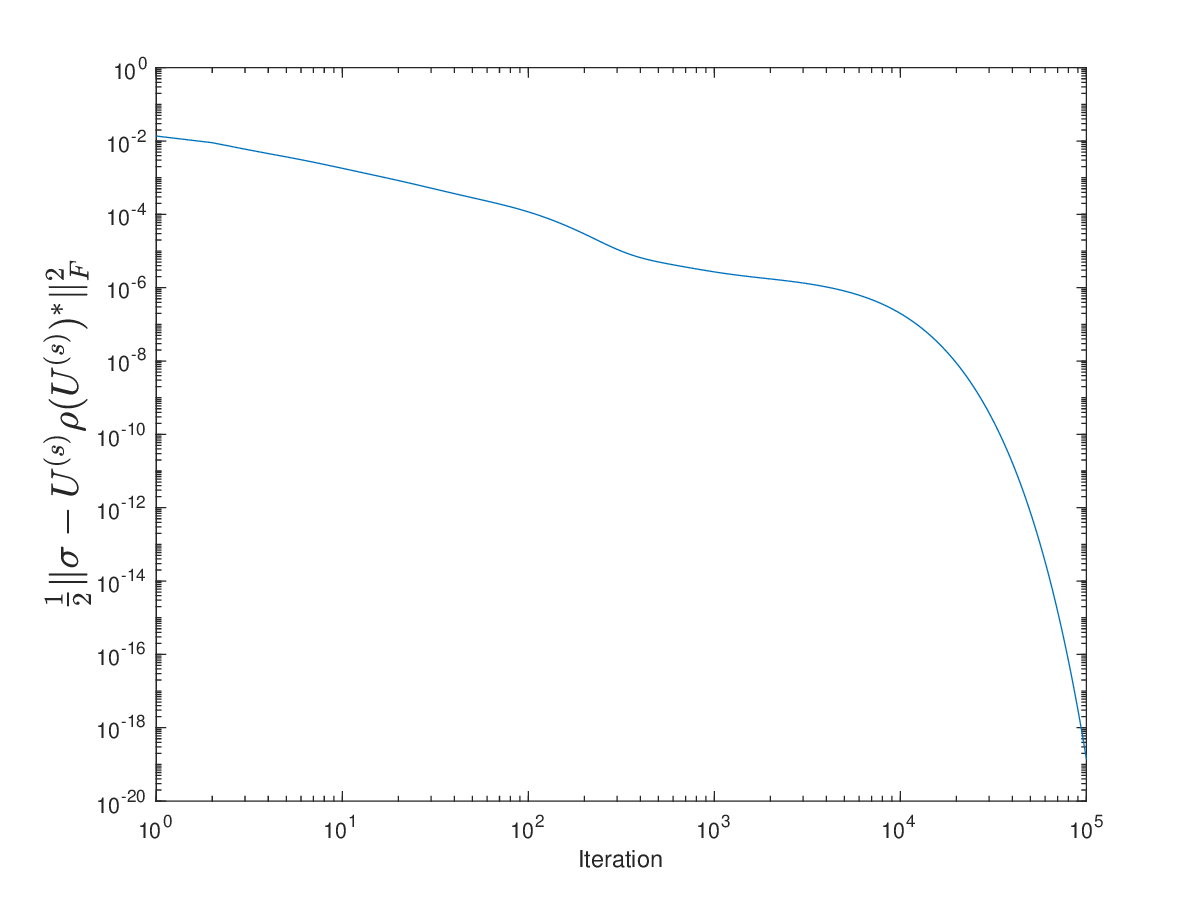}
        \caption{The objective value through iterations}
        \label{fig:example1_obj}
    \end{subfigure}
    \begin{subfigure}[t]{.5\textwidth}
        \centering
        \includegraphics[width=\linewidth]{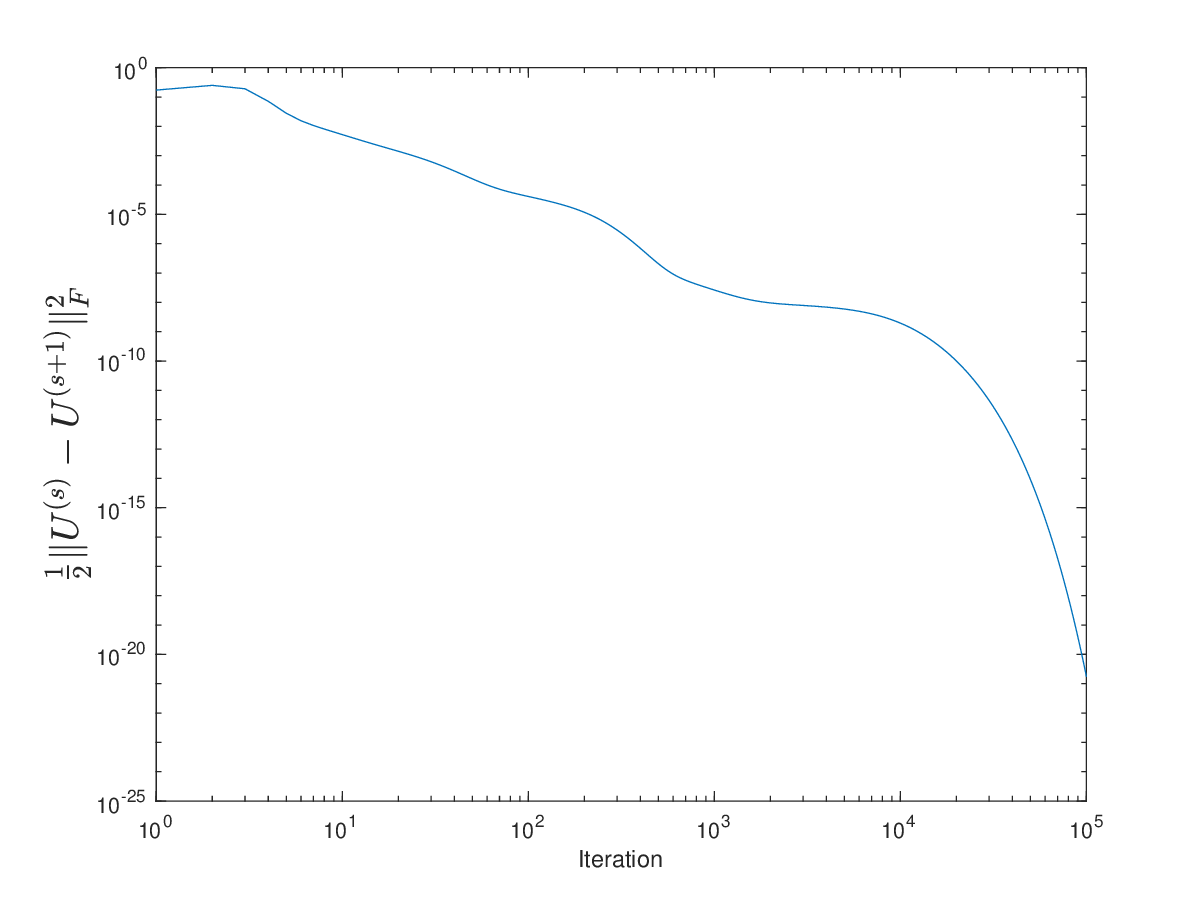}
        \caption{Difference between $U^{(s)}$ and $U^{(s+1)}$ within iterations.}
        \label{fig:example1_diffU}
    \end{subfigure}
    
    \caption{}
\end{figure}

Figure \ref{fig:example1_obj} shows the decreasing behavior of the objective function when applying our proposed algorithm to approximate a unitary quantum channel that transforms the input state $\rho$ into $\sigma$. We terminate the numerical process after 1000 iterations, as the objective value approaches $10^{-30}$, which can be consided zero, although the objective function decreases. Additionally, figure \ref{fig:example1_diffU} demonstrates the difference between the generated $U^{(s+1)}$ and the previous iteration $U^{(s)}$. This difference tends toward zero as the process nears termination. We also observe fluctuations during the iterations, but their magnitude is less than $10^{-20}$. Hence, we are confident that this example sufficiently supports our convergence theorem. Besides, we attribute these fluctuations to software or machine-induced errors.

Beyond demonstrating the effectiveness of the proposed algorithm in exploring a possible unitary quantum channel from a pair of quantum states, we also aim to highlight its capability to identify an appropriate quantum channel from a sequence of quantum pairs. To this end, we prepare 20 quantum input states $\{\rho_i\}_{i=1}^{20}\subset \mathbb{C}^{10\times 10}$ corresponding to observed states $\{\sigma_i\}_{i=1}^{20}$, where each pair is transformed by an unknown quantum channel $\hat{U}$.

\begin{figure}[!ht]
    \begin{subfigure}[t]{.5\textwidth}
        \centering
        \includegraphics[width=\linewidth]{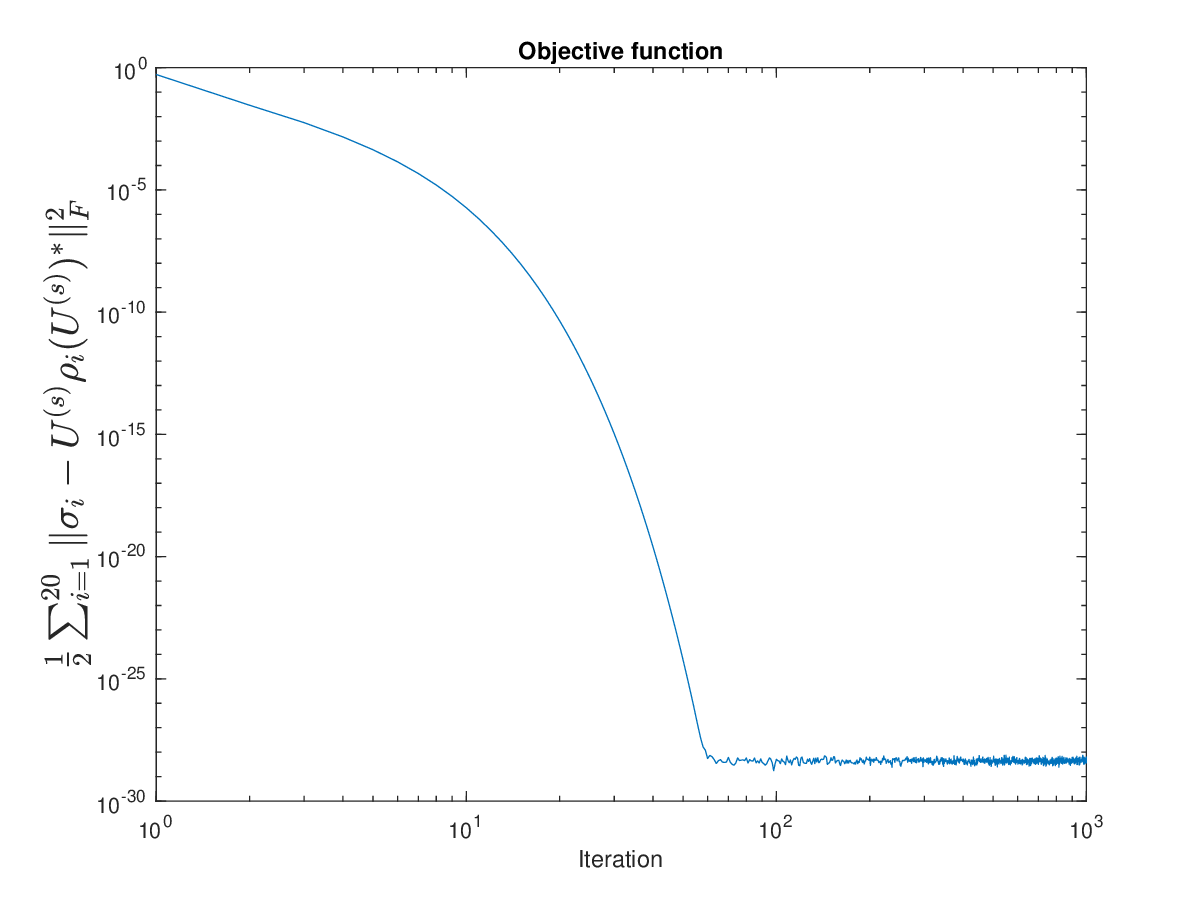}
        \caption{The objective value through iterations}
        \label{fig:example1b_obj}
    \end{subfigure}
    \begin{subfigure}[t]{.5\textwidth}
        \centering
        \includegraphics[width=\linewidth]{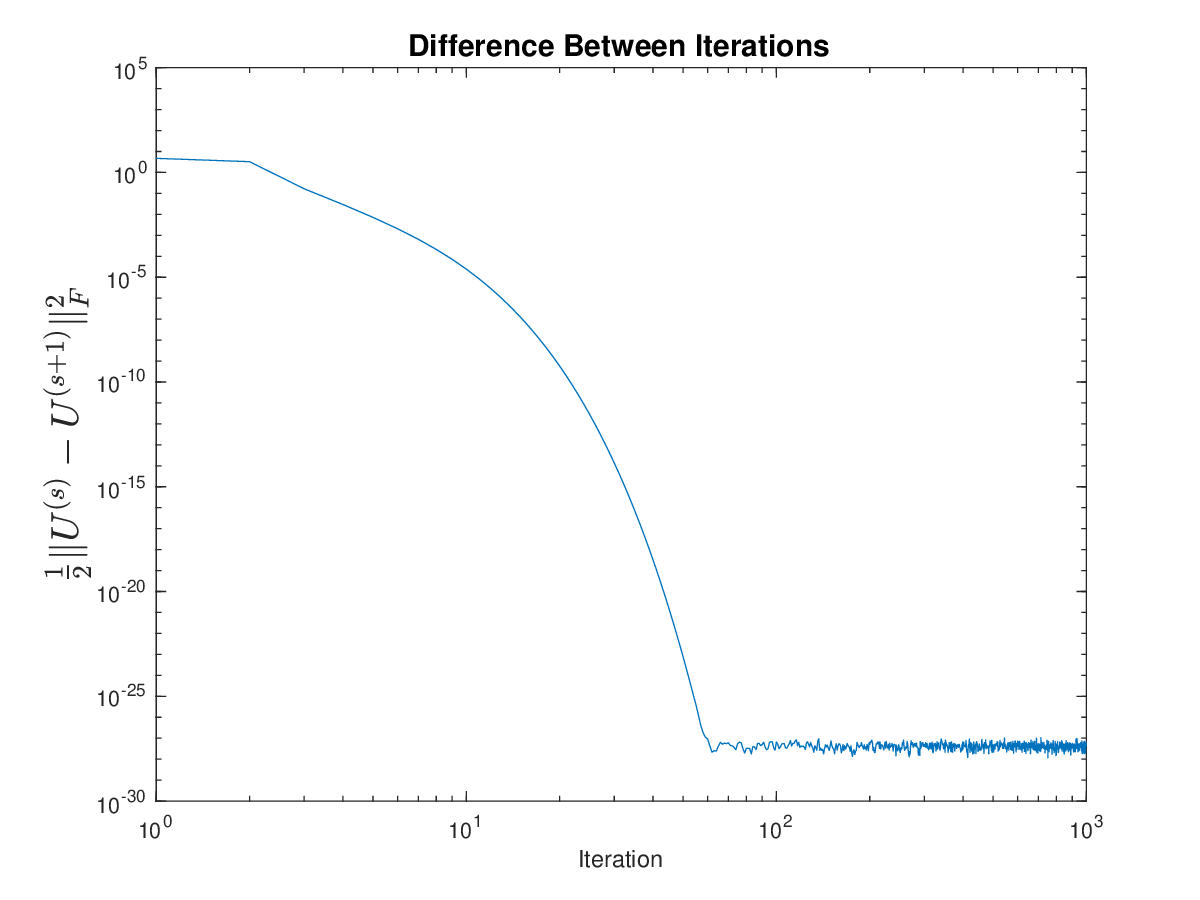}
        \caption{Difference between $U^{(s)}$ and $U^{(s+1)}$ within iterations.}
        \label{fig:example1b_diffU}
    \end{subfigure}
    
    \caption{}
\end{figure}
Figure~\ref{fig:example1b_obj} illustrates the decreasing trend of the objective function over iterations. Additionally, Figure~\ref{fig:example1b_diffU} shows that the difference between the generated $ U^{(s)} $ matrices gradually approaches zero. This example demonstrates that our proposed method can uncover the underlying mechanisms of a quantum process from numerous pairs of inputs and outputs.

\textbf{Example 2}
Identifying an unknown unitary quantum channel motivates this approximation framework. In this example, we will show our method's effectiveness in exploring the unknown quantum channel from given quantum pairs. To simplify our discussion, we assume noise-free output states. Assuming a sequence of output states is generated by the giving compositing quantum logic gate:
\begin{figure}[h]
    \centering
    \begin{subfigure}[t]{.5\textwidth}
        \includegraphics[width=\linewidth]{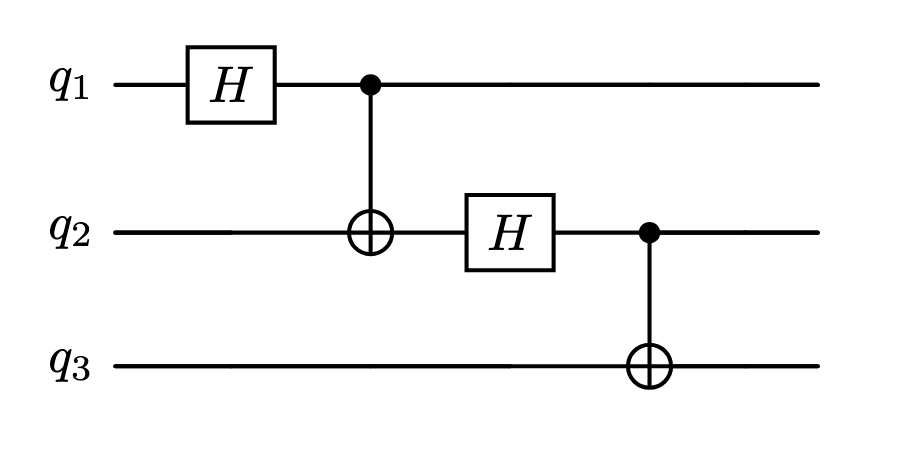}
    \end{subfigure}
    \caption{Synthetic quantum logic gate}
    \label{fig:example3}
\end{figure}
The $H$ gate in Figure \ref{fig:example3} represents the Hadamard gate, and the symbol \begin{quantikz}
\lstick{$q_1$} & \ctrl{1} & \qw \\
\lstick{$q_2$} & \targ{} & \qw
\end{quantikz} represents the control-not gate (CNOT gate).
The matrix representation of the given circuit is
\begin{eqnarray*}
    U = \dfrac{1}{2}\left[\begin{array}{cccccccc}
    1  & 0  & 1  & 0  & 1  & 0  & 1  & 0 \\
    0  & 1  & 0  & 1  & 0  & 1  & 0  & 1 \\
    0  & 1  & 0  & -1 & 0  & 1  & 0  & -1 \\
    1  & 0  & -1 & 0  & 1  & 0  & -1 & 0 \\
    1  & 0  & 1  & 0  & -1 & 0  & -1 & 0 \\
    0  & 1  & 0  & 1  & 0  & -1 & 0  & -1 \\
    0  & -1 & 0  & 1  & 0  & 1  & 0  & -1 \\
    -1 & 0  & 1  & 0  & 1  & 0  & -1 & 0
    \end{array}\right].
\end{eqnarray*}
We prepare a {non-degenerated quantum state $\rho$ and collect its output through the prescribed channel.} We first show in Figure \ref{fig:example3_obj} that the objective function is decreasing and eventually converges to $10^{-20}$ after $2000$ iteration steps. We denote the optimal solution $U'$ of the optimization problem
\eqref{eq:prob1}. Next, we assume that there is a diagonal matrix $D$ such that $U' = U(VDV^*)$, where $V$ collects the eigenvectors of $\rho$ in a given order. We find the elements in $D$ by following Step 3 of the reconstruction procedure outlined in Section 5. 
{
Finally, we reconstruct a unitary matrix $U'$, which differs from the ground truth $U$, a complex scalar with modulus 1. 

To demonstrate the effectiveness of our method, we run the example 20 times with different nondegenerate $\rho$. In each run, we record the difference between the normalized matrix $U$ (scaled by $1 / U_{11}$) and the normalized matrix $U'$ (scaled by $1 / U'_{11}$), namely $\left\|\frac{1}{U_{11}} U - \frac{1}{U'_{11}} U'\right\|_F$. The histogram of all 20 runs is shown in Figure \ref{fig:example3_hist}. The differences remain below $10^{-9}$, demonstrating that the computed result $U'$ from our method differs from the true result result $U$ only by a scalar with modulus 1.}
\begin{figure}[!ht]
    
        \centering
        \includegraphics[width=.5\linewidth]{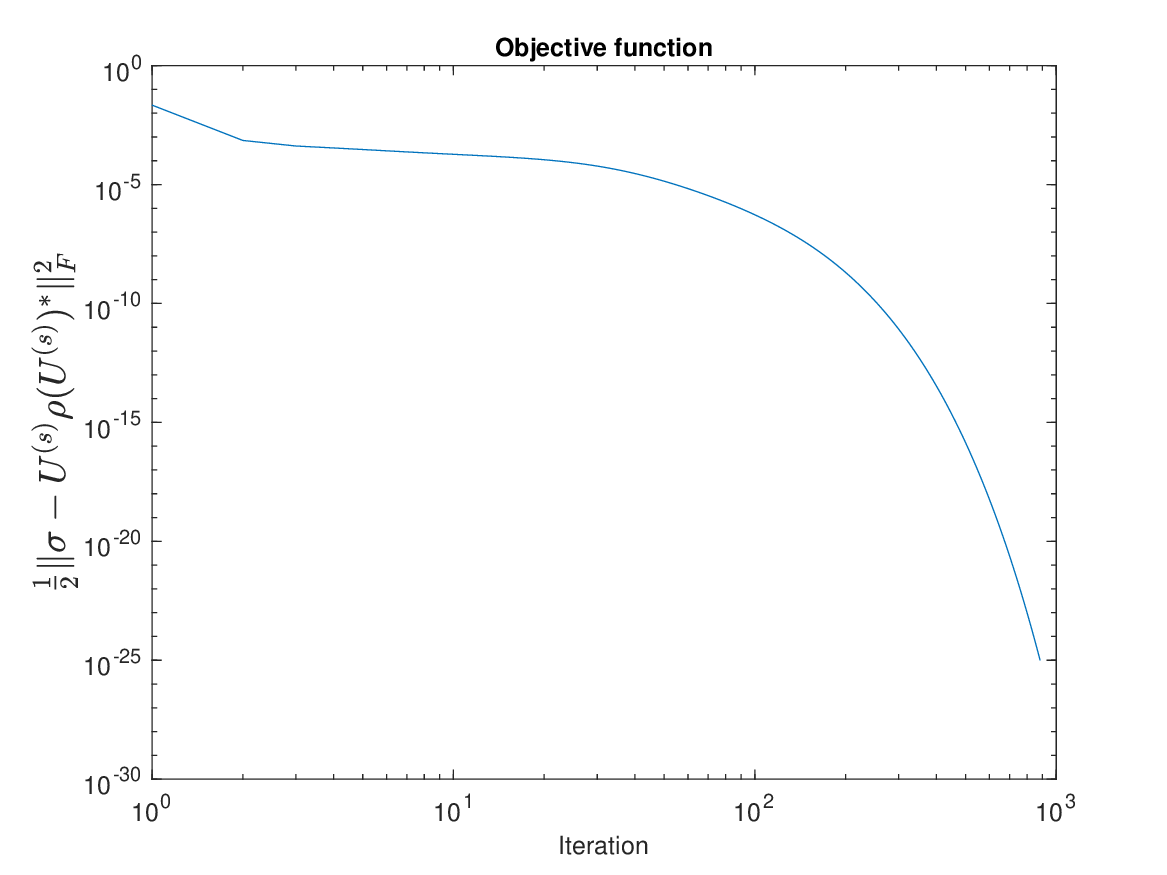}
        \caption{Objective function along iterations of one run}
        \label{fig:example3_obj}
\end{figure}
\begin{figure}[!ht]
        \centering
        \includegraphics[width=.5\linewidth]{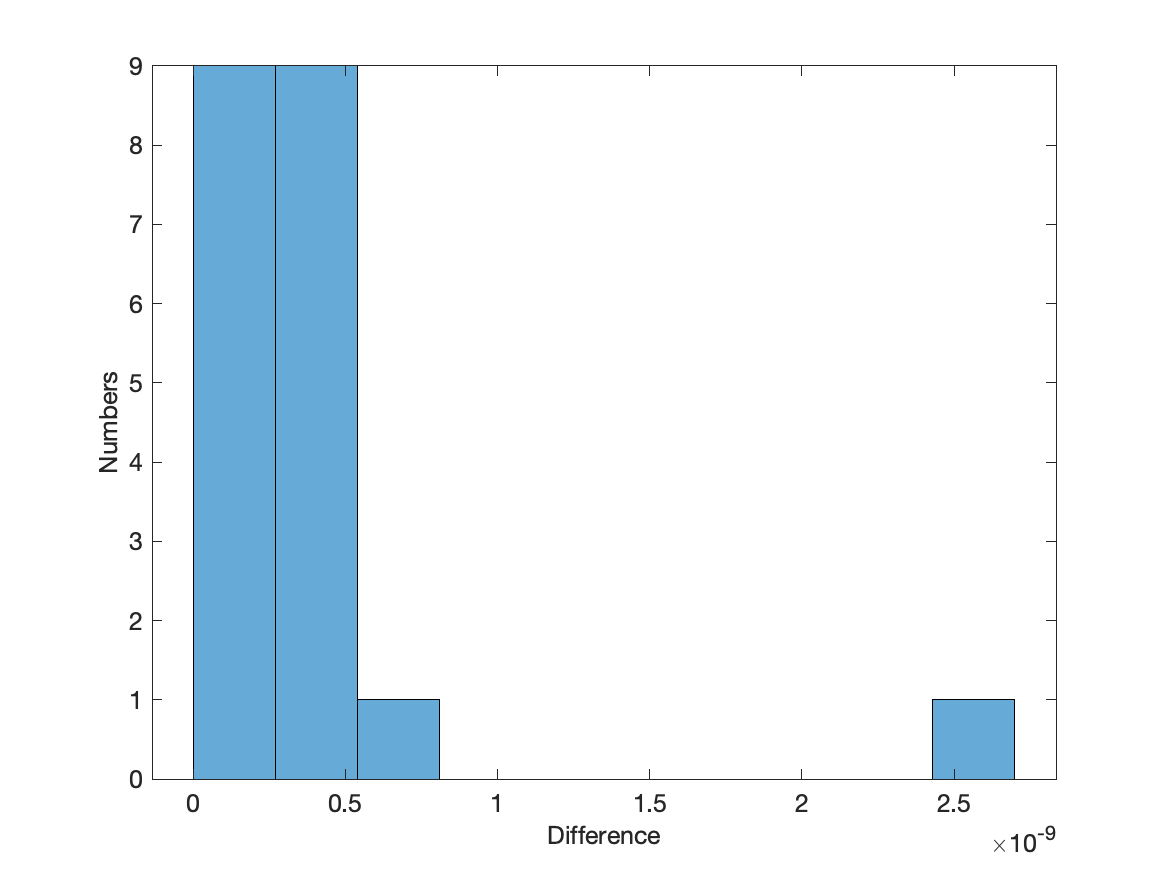}
        \caption{Histogram of the difference $\left\|\dfrac{1}{U_{11}}U-\dfrac{1}{U'_{11}}U'\right\|_F$ in 20 runs}
        \label{fig:example3_hist}
   
\end{figure}


\section{Conclusion}
{
This article focuses on characterizing a quantum channel represented by a single unitary operator using limited data. We begin by showing, through theoretical arguments, that the set of possible unitary matrices derived from limited data forms an equivalence class, which reduces the overall search space. We further prove that these unitary matrices, describing the channel action on all quantum states, differ only by a complex scalar of unit modulus. Building on this insight, we recast the exploration as an optimization problem and propose an effective algorithm based on polar decomposition to solve it. Our method generates a sequence that decreases the objective function and eventually converges to a local minimum through theoretical proofs and numerical experiments.  Consequently, we provide an efficient way to uncover the unknown unitary quantum channel and effectively leverage limited data.}

\section*{Declarations}

\begin{itemize}
\item Funding

The first author received support from the National Center for Theoretical Sciences of Taiwan and
the National Science and Technology Council through grant 112-2628-M-006 -009 -MY4.

The second author received support from the National Center for Theoretical Sciences of Taiwan and
the National Science and Technology Council through grant 110-2628-M-110 -002 -MY4

The third author received support from
the National Science and Technology Council through 113-2917-I-564-033.

\item Conflict of interest/Competing interests\\
The authors declare no potential conflict of interest.

\item Data availability\\
There is no associated data in this manuscript.
\end{itemize}

\noindent



\end{document}